\documentclass{jktrstyle}
\usepackage{amsfonts,amssymb,amsmath,epsfig,color}

\newtheorem{question}[theorem]{Question}

\begin{document}

\markboth{V.\ Bardakov, A.\ Vesnin, B.\ Wiest}
{Dynnikov coordinates on virtual braid groups}

%%%%%%%%%%%%%%%%%%%%% Publisher's Area please ignore %%%%%%%%%%%%%%
\catchline{}{}{}{}{}
%%%%%%%%%%%%%%%%%%%%%%%%%%%%%%%%%%%%%%%%%%%%%%%%%%%%%%%%%%%%%%%%%%%

\title{DYNNIKOV COORDINATES ON VIRTUAL BRAID GROUPS}

\author{VALERIY G. BARDAKOV and ANDREI YU. VESNIN}

\address{Sobolev Institute of Mathematics\\ pr. Koptyuga 4\\ 630090, Novosibirsk, Russia\\ bardakov@math.nsc.ru, vesnin@math.nsc.ru}

\author{BERT WIEST}
\address{IRMAR,  Universit\'e de Rennes 1,\\ Campus de Beaulieu\\ 35042 Rennes Cedex, France\\ bertold.wiest@univ-rennes1.fr}

\maketitle

\begin{abstract}
We define Dynnikov coordinates on virtual braid groups. We prove that they are faithful invariants of virtual 2-braids, and present evidence that they are also very powerful invariants for general virtual braids. 
\end{abstract}



%%%%%%%%%%%%%%%%%%%%%%%%%%%%%%%%%%%%%%%%%%%%%%%%%%%%%%%%%%%%%%%%%%%%%%%%%%%%

\section{Virtual braid groups}
\label{sec1}

The group of virtual braids  $VB_n$, $n\geqslant 2$, on $n$ strings was introduced by Kauffman~\cite{Kauffman} as a generalization of the classical braid group $B_n$. The most useful system of generators and defining relations of $VB_n$ was introduced by Vershinin in \cite{Vershinin}. The generators of $VB_n$ are  
\begin{equation}
\sigma_1, \ldots, \sigma_{n-1}, \rho_1, \ldots, \rho_{n-1}, \label{eq1}
\end{equation}
and the defining relations are the following: 
\begin{equation}
\sigma_i \, \sigma_{i+1} \, \sigma_i = \sigma_{i+1} \, \sigma_i \, \sigma_{i+1},  \qquad  \sigma_i \, \sigma_j = \sigma_j \, \sigma_i  \, \hbox{ \ if \ } |i-j| > 1,  \label{eq2}
\end{equation}
\begin{equation}
\rho_i^2 = 1, i=1, \ldots, n-1, \quad \rho_i \, \rho_{i+1} \, \rho_i = \rho_{i+1} \, \rho_i \, \rho_{i+1}, \quad \rho_i \, \rho_j = \rho_j \, \rho_i \hbox{ \ if \ }  |i-j| > 1, \label{eq3}
\end{equation}
\begin{equation}
\sigma_i \rho_j = \rho_j \sigma_i  \hbox{ \ \ if \ } |i-j| > 1, \qquad \rho_i \, \rho_{i+1} \, \sigma_i = \sigma_{i+1}  \, \rho_i  \, \rho_{i+1} .  \label{eq4}
\end{equation}
Thus, the group generated by $\sigma_1, \ldots, \sigma_{n-1}$ with relations (\ref{eq2}) is the braid group $B_n$; the group generated by $\rho_1, \ldots, \rho_{n-1}$ with relations (\ref{eq3}) is the symmetric group~$S_n$; the relations (\ref{eq4}) will be referred to as \emph{mixed relations}. The last presented relation is equivalent to 
\begin{equation}
\rho_{i+1} \, \rho_i \, \sigma_{i+1} = \sigma_i \, \rho_{i+1} \, \rho_i .
\end{equation}
We remark that the relations
\begin{equation}
\rho_i \, \sigma_{i+1} \, \sigma_i = \sigma_{i+1} \, \sigma_i \, \rho_{i+1}, \qquad \rho_{i+1} \, \sigma_i \, \sigma_{i+1}   =  \sigma_{i+1} \, \sigma_i \, \rho_i
\end{equation}
do not hold in $VB_n$, and these relations will be referred to as \emph{forbidden} relations. (Adding these forbidden relations 
yields the so-called \emph{braid permutation group} \cite{FRR}).

There is a natural epimorphism $\pi : VB_n \to S_n$ defined by 
$$
\pi (\sigma_i) = \pi (\rho_i) = \rho_i, \quad i=1, \ldots, n-1.  
$$
The kernel $\mbox{ker} (\pi)$ is called the \emph{virtual pure braid group} and is denoted by $VP_n$. Generators and relations for $VP_n$ are described in \cite{Bardakov}. It is easy to see that $VB_n$ is a semidirect product:  $VB_n = VP_n \rtimes S_n$.

\section{Coordinates on braid groups}
\label{sec2}

In \cite[Ch.~8]{book}, an action of the braid group $B_n$ on the integer  lattice $\mathbb Z^{2n}$  by piecewise-linear bijections is defined. 
For the reader's convenience we recall the definition. For $x \in \mathbb Z$ denote $x^+ = \max \{ 0, x \}$ and $x^- = \min \{ x, 0 \} $. 
Define actions 
$$
\sigma, \, \sigma^{-1} : \mathbb Z^4 \to \mathbb Z^4 
$$ 
on $(a,b,c,d) \in \mathbb Z^{4}$ as follows:
\begin{equation}
(a, b, c, d) \cdot  \sigma  = (a + b^+ + (d^+ - e)^+,  \, d - e^+,  \, c + d^- + (b^- + e)^-,  \, b + e^+), \label{action-by-sigma} 
\end{equation}
\begin{equation}
(a,b,c,d) \cdot  \sigma^{-1} = (a - b^+ - (d^+ + f)^+,  \, d + f^-,  \, c - d^-  - (b^- - f)^-,  \, b - f^-),  \label{action-by-sigma-inverse}
\end{equation}
where 
\begin{equation}
e = a - b^- - c + d^+, \qquad f = a + b^- - c - d^+ . \label{definition-e-f}
\end{equation}
For a given vector $(a_1, b_1, \ldots, a_n, b_n) \in \mathbb Z^{2n}$ we define the action by $\sigma_i^{\varepsilon} \in B_{n}$, where $i=1, \ldots, n-1$:
\begin{equation}
(a_1, b_1, \ldots, a_n, b_n) \cdot \sigma_i^{\varepsilon} = (a_1', b_1', \ldots, a_n', b_n'),  \label{action-by-generator-sigma-1}
\end{equation}
where $a_k' = a_k$, $b_k' = b_k$ if $k \neq i, i+1$, and 
\begin{equation}
(a_i', b_i', a_{i+1}', b_{i+1}') = \begin{cases}
\begin{array}{ll}
(a_i, b_i, a_{i+1}, b_{i+1}) \cdot \sigma, & \mbox{if } \, \varepsilon = 1, \cr  
(a_i, b_i, a_{i+1}, b_{i+1}) \cdot \sigma^{-1}, & \mbox{if } \, \varepsilon = -1. \cr  
\end{array}
\end{cases} \label{action-by-generator-sigma-2}
\end{equation}
For a word $w$ in the alphabet $\{ \sigma_1^{\pm 1}, \ldots, \sigma_{n-1}^{\pm 1}  \}$ we define an action by $w$: 
\begin{equation}
(a_1, b_1, \ldots  a_n, b_n) \cdot w = \begin{cases}
\begin{array}{ll}
(a_1, b_1, \ldots  a_n, b_n) , & \mbox{if } \, w= 1, \cr  
((a_1, b_1, \ldots  a_n, b_n)  \cdot \sigma_i^{\varepsilon}) \cdot w' , & \mbox{if } \, w = \sigma_i^{\varepsilon} w'. \cr  
\end{array}
\end{cases}
\end{equation}

It can be shown that the above action by $B_n$ on $\mathbb Z^{2n}$ is well defined, 
i.e. if two words $w_1$ and $w_2$ present the same element of the braid group $B_n$ then 
$$
(a_1, b_1, \ldots, a_n, b_n) \cdot w_1 = (a_1, b_1, \ldots, a_n, b_n) \cdot w_2 
$$ 
for any vector $(a_1, b_1, \ldots, a_n, b_n) \in \mathbb Z^{2n}$.   
By  the \emph{Dynnikov coordinates} of a braid we will mean the vector 
$(0,1, \ldots, 0, 1) \cdot w$, where $w$ is a word representing that braid. 

\begin{example}
Actions by some elements of $B_2$ on $(0,1,0,1) \in \mathbb Z^{2}$ are as follows: 
$$
\begin{array}{ll}
(0,1,0,1) \cdot  \sigma_1 = (1,0,0,2), & (0,1,0,1) \cdot  \sigma_1^{-1}  = (-1,0,0,2), \cr 
(0,1,0,1) \cdot  \sigma_1^2 = (1,-1,0,3), & (0,1,0,1) \cdot \sigma_1^{-2}  = (-1,-1,0,3), \cr 
\hspace{2,6cm}\vdots & \hspace{2,75cm}\vdots\cr
(0,1,0,1) \cdot \sigma_1^k = (1,-k+1,0,k+1), \phantom{xx} & (0,1,0,1) \cdot \sigma_1^{-k}  = (-1,-k+1,0,k+1), 
\end{array}
$$
where $k\in\mathbb N$. 
Also, acting by some elements of $B_{3}$ on $(0,1,0,1,0,1) \in \mathbb Z^{6}$ we have
$$
\begin{array}{l}
(0,1,0,1,0,1) \cdot \sigma_1 \sigma_2^{-1} = (1,0,-2,0,0,3), \cr 
(0,1,0,1,0,1) \cdot \sigma_1 \sigma_2 \sigma_1 = (2,0,1,0,0,3) .
\end{array}
$$
\end{example}

\begin{remark}
It is also shown in \cite{book} that Dynnikov coordinates are faithful invariants of braids, i.e.\ if $(0,1, \ldots, 0,1)  \cdot w_1 = (0,1, \ldots, 0,1) \cdot w_2$ then $w_1 = w_2$ in $B_n$; thus, Dynnikov coordinates are very useful for solving the words problem in $B_n$.  

Here is an outline of the proof: there is a bijection between vectors 
in~$\mathbb Z^{2n}$ and integer laminations of a sphere with $n+3$ 
punctures $P_0,P_1,\ldots,P_{n+1},P_\infty$; under this bijection, our 
action of~$B_n$ on $\mathbb Z^{2n}$ corresponds to the 
$B_n$-action on a disk containing punctures $P_1,\ldots,P_n$. The key 
observation is that this disk is \emph{filled} by the lamination encoded 
by the vector $(0,1, \ldots, 0,1)$ (i.e.\ cutting this disk along its 
intersection with the lamination yields only disks and once-punctured disks).
\end{remark}

\section{Coordinates on virtual braid groups}
\label{sec3}

Let us define an action by elements of $VB_{n}$  on $\mathbb Z^{2n}$. Consider the actions on $\mathbb Z^{4}$ by $\sigma$ and $\sigma^{-1}$ as defined in (\ref{action-by-sigma}) and (\ref{action-by-sigma-inverse}), and define the action by $\rho$ as the following permutation of coordinates:  
\begin{equation}
(a, b, c, d) \cdot  \rho = (c, d, a, b). \label{action-by-rho}
\end{equation}
For a given vector $(a_1, b_1, \ldots, a_n, b_n) \in \mathbb Z^{2n}$ we define the action by $\rho_{i} \in VB_{n}$, $i=1, \ldots, n-1$:
\begin{equation}
(a_1, b_1, \ldots, a_n, b_n) \cdot  \rho_i = (a_1', b_1', \ldots, a_n', b_n'), 
\end{equation}
where $a_k' = a_k$, $b_k' = b_k$ for $k \neq i, i+1$, and
\begin{equation}
(a_i', b_i', a_{i+1}', b_{i+1}') =  (a_i, b_i, a_{i+1}, b_{i+1} ) \cdot  \rho.
\end{equation}
The action by $\sigma_{i}^{\varepsilon} \in VB_{n}$ on $\mathbb Z^{2n}$ is defined according to (\ref{action-by-generator-sigma-1}) and (\ref{action-by-generator-sigma-2}). 

Suppose that $w$ is a word in the alphabet $\{ \sigma_1^{\pm1}, \ldots, \sigma_{n-1}^{\pm 1}, \rho_1, \ldots, \rho_{n-1}\}$ representing an element of the group $VB_n$. Then we define
\begin{equation}
(a_1, b_1, \ldots, a_n, b_n) \cdot w = \begin{cases} \begin{array}{ll}
( ( a_1, b_1, \ldots, a_n, b_n) \cdot \rho_i ) \cdot w', & \mbox{if } \, w = \rho_i w' , \cr
( ( a_1, b_1, \ldots, a_n, b_n) \cdot \sigma_i^{\varepsilon} ) \cdot w', & \mbox{if} \, w = \sigma_i^{\varepsilon} w' . \cr
\end{array}
\end{cases}
\end{equation}

To show that the action by $VB_{n}$ on $\mathbb Z^{2n}$ is correctly defined we will verify that the defining relations of the group $VB_n$ are satisfied. Since $\rho_i$ acts by permuting pairs of coordinates, the relations of the group $S_n$ are obviously satisfied. The fact that relations of the group $B_n$ are satisfied follows from \cite{book}. So, we need to check only the case of mixed relations, i.e. that for any $v \in \mathbb Z^{2n}$ the relations
\begin{equation}
v \cdot ( \sigma_i \, \rho_j ) = v \cdot (\rho_j \, \sigma_i), \quad | i - j | > 1, \label{eqn_mixed1}
\end{equation}
\begin{equation}
v \cdot ( \rho_i \, \rho_{i+1} \, \sigma_i ) = v \cdot ( \sigma_{i+1} \, \rho_i \, \rho_{i+1} )  \label{eqn_mixed2}
\end{equation}
hold.  Relations (\ref{eqn_mixed1}) hold obiously, because $\sigma_i$ acts non-trivial only on the subvector $(a_i, b_i, a_{i+1}, b_{i+1})$ and $\rho_j$ acts non-trivially only on the subvector $(a_j, b_j, a_{j+1}, b_{j+1})$. In order to verify (\ref{eqn_mixed2})  it is enough to consider the case $i=1$ in the group $VB_3$. Denote 
$$
(x,y,z,t) \cdot \sigma = (a'(x,y,z,t), b'(x,y,z,t), c'(x,y,z,t), d'(x,y,z,t)) 
$$
From  
$$
\begin{gathered}
(a_1, b_1, a_2, b_2, a_3, b_3) \cdot (\rho_1 \, \rho_2 \, \sigma_1)  =  (a_2, b_2, a_1, b_1, a_3, b_3) \cdot (\rho_2 \, \sigma_1) \cr  =  (a_2, b_2, a_3, b_3, a_1, b_1) \cdot \sigma_1  =   ((a_2, b_2, a_3, b_3) \cdot \sigma, a_1, b_1)
\end{gathered}
$$
and 
$$
\begin{gathered}
(a_1, b_1, a_2, b_2, a_3, b_3) \cdot (\sigma_2 \, \rho_1 \, \rho_2)  =  (a_1, b_1, (a_2, b_2, a_3, b_3) \cdot  \sigma ) \cdot \rho_1 \, \rho_2 \cr = (a' (a_2, b_2, a_3, b_3), b' (a_2, b_2, a_3, b_3), a_1, b_1, c' (a_2, b_2, a_3, b_3),  d' (a_2, b_2, a_3, b_3)) \cdot \rho_2 \cr  =  (a' (a_2, b_2, a_3, b_3), b' (a_2, b_2, a_3, b_3), c' (a_2, b_2, a_3, b_3), d' (a_2, b_2, a_3, b_3), a_1, b_1)  \cr = ((a_2, b_2, a_3, b_3) \cdot \sigma, a_1, b_1)
\end{gathered}
$$
we see that (\ref{eqn_mixed2}) holds. 

\begin{example}
Actions by some elements of $V B_2$  on $(0,1,0,1) \in \mathbb Z^{4}$ are as follows: 
$$
\begin{array}{l}
(0,1,0,1) \cdot \sigma_1 \rho_1 = (0,2,1,0),  \cr 
(0,1,0,1) \cdot \sigma_1 \rho_{1} \sigma_{1} = (3,0,0,2) , \cr 
(0,1,0,1) \cdot \sigma_1 \rho_{1} \sigma_{1}^{-1} = (-2,-1,1,3).  \cr 
\end{array}
$$
\end{example}

Let us demonstrate that the forbidden relations are not satisfied. More exactly, we show that for $v = (0,1,0,1,0,1)$ we get
\begin{equation}
v \cdot (\rho_1 \, \sigma_2 \, \sigma_1 ) \neq v \cdot (\sigma_2 \, \sigma_1 \, \rho_2 ) , \label{eqn_bad1}
\end{equation}
\begin{equation}
v \cdot (\rho_2 \, \sigma_1 \, \sigma_2 ) \neq v \cdot (\sigma_1 \, \sigma_2 \, \rho_1 ) . \label{eqn_bad2}
\end{equation}
Indeed, (\ref{eqn_bad1}) holds because 
$$
(0,1,0,1,0,1) \cdot (\rho_1 \, \sigma_2 \, \sigma_1) = (0,1,0,1,0,1) \cdot \sigma_2 \, \sigma_1 = (0,1,1,0,0,2) \cdot \sigma_1 = (2,0,0,1,0,2), 
$$
but
$$
(0,1,0,1,0,1) \cdot (\sigma_2 \, \sigma_1 \, \rho_2) = (2,0,0,1,0,2) \cdot \rho_2 = (2,0,0,2,0,1).
$$
Analogously, (\ref{eqn_bad2}) holds because 
$$
(0,1,0,1,0,1) \cdot (\rho_2 \, \sigma_1 \, \sigma_2) = (0,1,0,1,0,1) \cdot \sigma_1 \, \sigma_2 = (1,0,0,2,0,1) \cdot \sigma_2 = (1,0,2,0,0,3), 
$$
but
$$
(0,1,0,1,0,1) \cdot (\sigma_1 \, \sigma_2 \, \rho_1) = (1,0,2,0,0,3) \cdot \rho_1 = (2,0,1,0,0,3).
$$

\begin{question} \label{questManturov} 
Is there any relation between our coordinates on $VB_n$ and the invariant
of virtual braids defined by Manturov in \cite{Manturov}?
\end{question}

\section{Faithfulness of the $VB_{n}$-action on $\mathbb Z^{2n}$} 

In this section we will be concerned with the following

\begin{question} \label{QuestFaithful} 
Is the $VB_n$-action on $\mathbb Z^{2n}$ faithful? In other words, is it true that only the trivial element  of~$VB_n$ acts as the identity on $\mathbb Z^{2n}$?
\end{question}

In computer experiments, we have tested several billion ($10^9$) random virtual braids with 3, 4 and 5 strands, but failed to  find one that would provide a negative answer. The programs used for these tests (written in Scilab) can be obtained from B.~Wiest's web page \cite{Wwebpage}.

It should be stressed that nontrivial elements of $VB_n$ may very well act trivially on individual vectors. Let us, for instance, look at the $VB_3$-action on $(0, 1, 0, 1, 0, 1)$. The actions of $\rho_1$ and $\rho_2$ obviously fix this vector, but those  of many other braids do, too. Here is one particularly striking example: 

\begin{example}
The virtual 3 strand braid 
%[1+%i  2   1+%i  2-%i  1+%i  2+%i  1-%i  1   2+%i    1   1+%i  2   1-%i  2   2-%i  1-%i  2+%i  1-%i  2   1-%i]$$
$$
\beta=\sigma_1 \rho_2 \sigma_1 \sigma_2^{-1} \sigma_1 \sigma_2 \sigma_1^{-1} \rho_1 \sigma_2 \rho_1 \sigma_1 \rho_2 \sigma_1^{-1} \rho_2 \sigma_2^{-1} \sigma_1^{-1} \sigma_2 \sigma_1^{-1} \rho_2 \sigma_1^{-1}
$$ 
acts trivially on the vector $(0, 1, 0, 1, 0, 1)$, and indeed in computer experiments, we found that it acted nontrivially on only about 0.{\small 25}\% of randomly generated vectors of $\mathbb Z^6$ with integer coefficients between $-100$ and $100$. In this sense, $\beta$ is ``nearly a negative answer'' to Question \ref{QuestFaithful}.
Another one is the virtual braid  $\sigma_2^{-1}  \sigma_1  \rho_2  \sigma_2  \sigma_1  \sigma_2^{-1}  \rho_2  \sigma_1   \rho_2  \sigma_2 \rho_1  \sigma_2^{-1}  \rho_1  \sigma_1^{-1}  \sigma_2^{-1}  \rho_2  \sigma_1^{-1}  \sigma_2$ which also moves only about 0.{\small 6}\% of random vectors.
\end{example}

We now turn our attention to the case $n=2$. 

\begin{example}
It is known \cite{Manturov} that  $\beta = (\sigma_{1}^{2} \rho_{1} \sigma_{1}^{-1} \rho_{1} \sigma_{1}^{-1} \rho_{1})^{2} \in VB_{2}$ and the  identity element ${\bf e} \in VB_{2}$ cannot be distinguished by the Burau representation. Considering actions by these elements on $(0,1,0,1)$ we get  $(0,1,0,1) \cdot  \beta = (85, 49, -90, -47)$, while $(0,1,0,1) \cdot {\bf e} = (0, 1, 0, 1)$.  Thus, these elements are distinguished by Dynnikov coordinates.
\end{example}

In fact, in the case $n=2$ we have a positive answer to Question~\ref{QuestFaithful}. 

\begin{theorem}\label{T:VB2}
The $VB_{2}$--action on $\mathbb Z^{4}$ given by the above formulae is faithful.
\end{theorem}

\begin{proof}
We will, in fact show a stronger result than Theorem~\ref{T:VB2}, namely that the only element of $VB_2$ which acts trivially  on a vector $(0,x,0,y)$, where $x$ and $y$ are different positive integers, is the trivial one. As a simplest vector of this type one can take $(0,2,0,1)$. 

Consider the set of symbols $S = \{ 0, +, -, +0, -0 \}$. With $s \in S$ we associate the following subset of $\mathbb Z$:
\begin{eqnarray*}
& s = 0 
 \longleftrightarrow   \{ x \in \mathbb Z \, : \, x = 0 \}, & \cr
& s = +  \longleftrightarrow  \{ x \in \mathbb Z \, : \, x > 0 \}, \quad 
s = - \longleftrightarrow \{ x \in \mathbb Z \, : \, x < 0 \}, & \cr
& s = +0 \longleftrightarrow  \{ x \in \mathbb Z \, : \, x \geqslant 0 \}, \quad
s = -0  \longleftrightarrow  \{ x \in \mathbb Z \, : \, x \leqslant 0 \}.   &
\end{eqnarray*}

A quadruple  $(s_{1}, s_{2}, s_{3}, s_{4})$, where $s_{i} \in S$, indicates the set all quadruples $(a,b,c,d) \in \mathbb Z^{4}$ such that $a$, $b$, $c$, $d$ belongs to subset of $\mathbb Z$ associated with symbols $s_{1}$, $s_{2}$, $s_{3}$, $s_{4}$, respectively. For example, 
$$
(+, +0, 0, -) = \{ (a,b,c,d) \in \mathbb Z^{4} \, : \, a>0, \, b\geqslant 0, \, c=0, \, d<0 \} .
$$
For instance, the vector $(-1, -1, 0, 4)$ belongs to sets associated with $(-,-,0,+)$ and $(-,-0,+0,+)$, as well as with some other quadruples of symbols.    

Let $\sigma, \sigma^{-1}, \rho  : \mathbb Z^{4} \to \mathbb Z^{4}$ be the transformations defined by (\ref{action-by-sigma}), (\ref{action-by-sigma-inverse}), (\ref{action-by-rho}).  Let us apply a sequence of such transformations (without fragments $ \sigma \sigma^{-1} $, or $\sigma^{-1} \sigma$,  or $\rho \rho$) to the initial vector, say  $(0,2,0,1)$. The statement of the theorem will follow from the fact that we trace out a path in the diagram shown in Figure~\ref{F:VB2}, where an arrow with labels $\beta$ and $i. \beta$, where $i \in \{ 1, 2, \ldots, 14 \}$ and $\beta \in \{ \sigma, \sigma^{-1}, \rho \}$ from box $(s_{1}, s_{2}, s_{3}, s_{4})$ to box $(t_{1}, t_{2}, t_{3}, t_{4})$ means that the image of any element of the subset of $\mathbb Z^{4}$ associated with $(s_{1}, s_{2}, s_{3}, s_{4})$ under tranformation $\beta$ belongs to the subset of $\mathbb  Z^{4}$ associated with $(t_{1}, t_{2}, t_{3}, t_{4})$. The numbering of arrows by $1, 2, \ldots, 14$ is done for the reader's convenience in the further discussion of cases. 

 %%%
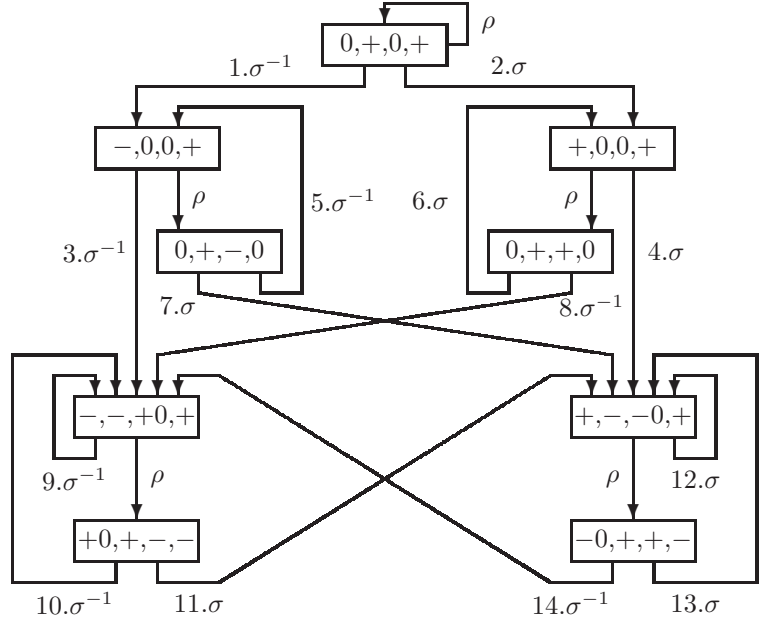
\begin{figure}[ht]
\begin{center}
\unitlength=.55mm
\begin{picture}(0,150)
\thicklines
\put(25,140){\makebox(0,0){$\rho$}}
\qbezier(15,135)(15,135)(20,135)
\qbezier(20,135)(20,135)(20,145)
\qbezier(20,145)(20,145)(0,145)
\put(0,145){\vector(0,-1){5}}
\qbezier(-15,130)(-15,130)(15,130)
\qbezier(-15,140)(-15,140)(15,140)
\qbezier(-15,130)(-15,130)(-15,140)
\qbezier(15,130)(15,130)(15,140)
\put(0,135){\makebox(0,0){0,+,0,+}}
\qbezier(40,105)(40,105)(70,105)
\qbezier(40,115)(40,115)(70,115)
\qbezier(40,105)(40,105)(40,115)
\qbezier(70,105)(70,105)(70,115)
\put(55,110){\makebox(0,0){+,0,0,+}}
\put(50,120){\vector(0,-1){5}}
\put(60,120){\vector(0,-1){5}}
\qbezier(-40,105)(-40,105)(-70,105)
\qbezier(-40,115)(-40,115)(-70,115)
\qbezier(-40,105)(-40,105)(-40,115)
\qbezier(-70,105)(-70,105)(-70,115)
\put(-55,110){\makebox(0,0){$-$,0,0,+}}   %
\put(-50,120){\vector(0,-1){5}}
\put(-60,120){\vector(0,-1){5}}
\qbezier(5,130)(5,130)(5,125)
\qbezier(5,125)(5,125)(60,125)
\qbezier(60,125)(60,125)(60,115)
\put(30,130){\makebox(0,0){$2. \sigma$}}
\qbezier(-5,130)(-5,130)(-5,125)
\qbezier(-5,125)(-5,125)(-60,125)
\qbezier(-60,125)(-60,125)(-60,115)
\put(-30,130){\makebox(0,0){$1. \sigma^{-1}$}}
\qbezier(25,80)(25,80)(55,80)
\qbezier(25,90)(25,90)(55,90)
\qbezier(25,80)(25,80)(25,90)
\qbezier(55,80)(55,80)(55,90)
\put(40,85){\makebox(0,0){0,+,+,0}}
\qbezier(50,105)(50,105)(50,90)
\put(50,95){\vector(0,-1){5}}
\put(45,97.5){\makebox(0,0){$\rho$}}
\qbezier(60,105)(60,105)(60,50)
\put(68,85){\makebox(0,0){$4. \sigma$}}
\qbezier(30,80)(30,80)(30,75)
\qbezier(30,75)(30,75)(20,75)
\qbezier(20,75)(20,75)(20,120)
\qbezier(20,120)(20,120)(50,120)
\qbezier(50,120)(50,120)(50,115)
\put(11,97.5){\makebox(0,0){$6. \sigma$}}
\qbezier(-25,80)(-25,80)(-55,80)
\qbezier(-25,90)(-25,90)(-55,90)
\qbezier(-25,80)(-25,80)(-25,90)
\qbezier(-55,80)(-55,80)(-55,90)
\put(-40,85){\makebox(0,0){0,+,$-$,0}}   %
\qbezier(-50,105)(-50,105)(-50,90)
\put(-50,95){\vector(0,-1){5}}
\put(-45,97.5){\makebox(0,0){$\rho$}}
\qbezier(-60,105)(-60,105)(-60,50)
\put(-70,85){\makebox(0,0){$3. \sigma^{-1}$}}
\qbezier(-30,80)(-30,80)(-30,75)
\qbezier(-30,75)(-30,75)(-20,75)
\qbezier(-20,75)(-20,75)(-20,120)
\qbezier(-20,120)(-20,120)(-50,120)
\qbezier(-50,120)(-50,120)(-50,115)
\put(-10,97.5){\makebox(0,0){$5. \sigma^{-1}$}}
%%%
\put(50,72){\makebox(0,0){$8. \sigma^{-1}$}}
\qbezier(45,80)(45,80)(45,75)
\qbezier(45,75)(45,75)(-55,60)
\qbezier(-55,60)(-55,60)(-55,50)
\put(-50,72){\makebox(0,0){$7. \sigma$}}
\qbezier(-45,80)(-45,80)(-45,75)
\qbezier(-45,75)(-45,75)(55,60)
\qbezier(55,60)(55,60)(55,50)
\put(-50,55){\vector(0,-1){5}}
\put(-55,55){\vector(0,-1){5}}
\put(-60,55){\vector(0,-1){5}}
\put(-65,55){\vector(0,-1){5}}
\put(-70,55){\vector(0,-1){5}}
\qbezier(-75,40)(-75,40)(-45,40)
\qbezier(-75,50)(-75,50)(-45,50)
\qbezier(-75,40)(-75,40)(-75,50)
\qbezier(-45,40)(-45,40)(-45,50)
\put(-60,45){\makebox(0,0){$-$,$-$,+0,+}}   %
\put(50,55){\vector(0,-1){5}}
\put(55,55){\vector(0,-1){5}}
\put(60,55){\vector(0,-1){5}}
\put(65,55){\vector(0,-1){5}}
\put(70,55){\vector(0,-1){5}}
\qbezier(75,40)(75,40)(45,40)
\qbezier(75,50)(75,50)(45,50)
\qbezier(75,40)(75,40)(75,50)
\qbezier(45,40)(45,40)(45,50)
\put(60,45){\makebox(0,0){+,$-$,$-$0,+}}   %
\put(-55,30){\makebox(0,0){$\rho$}}
\put(-60,40){\vector(0,-1){20}}
\qbezier(-75,10)(-75,10)(-45,10)
\qbezier(-75,20)(-75,20)(-45,20)
\qbezier(-75,10)(-75,10)(-75,20)
\qbezier(-45,10)(-45,10)(-45,20)
\put(-60,15){\makebox(0,0){+0,+,$-$,$-$}}   %
\put(55,30){\makebox(0,0){$\rho$}}
\put(60,40){\vector(0,-1){20}}
\qbezier(75,10)(75,10)(45,10)
\qbezier(75,20)(75,20)(45,20)
\qbezier(75,10)(75,10)(75,20)
\qbezier(45,10)(45,10)(45,20)
\put(60,15){\makebox(0,0){$-$0,+,+,$-$}}    %
\qbezier(70,40)(70,40)(70,35)
\qbezier(70,35)(70,35)(80,35)
\qbezier(80,35)(80,35)(80,55)
\qbezier(80,55)(80,55)(70,55)
\put(75,30){\makebox(0,0){$12. \sigma$}}
\qbezier(65,10)(65,10)(65,5)
\qbezier(65,5)(65,5)(90,5)
\qbezier(90,5)(90,5)(90,60)
\qbezier(90,60)(90,60)(65,60)
\qbezier(65,60)(65,60)(65,55)
\put(75,0){\makebox(0,0){$13. \sigma$}}
\qbezier(-70,40)(-70,40)(-70,35)
\qbezier(-70,35)(-70,35)(-80,35)
\qbezier(-80,35)(-80,35)(-80,55)
\qbezier(-80,55)(-80,55)(-70,55)
\put(-75,30){\makebox(0,0){$9. \sigma^{-1}$}}
\qbezier(-65,10)(-65,10)(-65,5)
\qbezier(-65,5)(-65,5)(-90,5)
\qbezier(-90,5)(-90,5)(-90,60)
\qbezier(-90,60)(-90,60)(-65,60)
\qbezier(-65,60)(-65,60)(-65,55)
\put(-75,0){\makebox(0,0){$10. \sigma^{-1}$}}
\qbezier(-55,10)(-55,10)(-55,5)
\qbezier(-55,5)(-40,5)(-40,5)
\qbezier(-40,5)(-40,5)(40,55)
\qbezier(40,55)(40,55)(50,55)
\put(-45,0){\makebox(0,0){$11. \sigma$}}
\qbezier(55,10)(55,10)(55,5)
\qbezier(55,5)(40,5)(40,5)
\qbezier(40,5)(40,5)(-40,55)
\qbezier(-40,55)(-40,55)(-50,55)
\put(45,0){\makebox(0,0){$14. \sigma^{-1}$}}
\end{picture}
\end{center}
\caption{The diagram of actions of $VB_2$.
}
\label{F:VB2}
\end{figure}
%%%

It is easy to see from (\ref{action-by-sigma}), (\ref{action-by-sigma-inverse}) and (\ref{action-by-rho}) that for any $\beta \in VB_{2}$ there is the following invariant: if $(a^{*}, b^{*}, c^{*}, d^{*}) = (a,b,c,d) \cdot \beta$ then $b^{*} + d^{*} = b + d$. In particular, if the initial vector is taken to be $(0,2,0,1)$ then for any vector in the diagram the sum of its second and fourth coordinates is equal to 3. 

Below to simplify expressions we will use the following notations:
$$
(a', b', c', d' ) = (a, b, c, d) \cdot \sigma, \qquad (a'', b'', c'', d'') = (a, b, c, d) \cdot \sigma^{-1}.
$$
By (\ref{action-by-sigma} ) and (\ref{action-by-sigma-inverse}) we have 
\begin{equation}
\begin{cases} 
\begin{array}{l}
a' = a + b^+ + (- a + b^{-} + c)^+,  \cr 
b' = d - (a - b^{-} - c + d^{+})^+, \cr 
c' = c + d^-  + (a - c + d^{+})^-, \cr 
d' = b + (a - b^{-} - c + d^{+})^+, \cr 
\end{array}
\end{cases}
\qquad 
\begin{cases} 
\begin{array}{l}
a'' = a - b^+ - (a + b^{-} - c)^+,  \cr 
b'' = d + (a + b^{-} - c - d^{+})^-, \cr 
c'' = c - d^-  - (-a + c + d^{+})^-, \cr 
d'' = b - (a + b^{-} - c - d^{+})^-, \cr  
\end{array}
\end{cases} \label{actions}
\end{equation}

First of all we remark that the arrows of the diagram related to the action by $\rho$ hold obviously.  The other actions will be considered case by case according to the numbering of the arrows. 

\underline{1.}  Consider the action by $\sigma^{-1}$ on $(0,+,0,+)$. Each element of $(0,+,0,+)$ has a form $(0,b,0,d)$ for some $b, d > 0$. Since its image $(0,b,0,d) \cdot \sigma^{-1} = (-b, 0, 0, b+d)$ belongs to  $(-,0,0,+)$, the corresponding arrow of the diagram is proven. 

\underline{2.}  Consider the action by $\sigma$ on $(0,+,0,+)$. Each element of $(0,+,0,+)$ has a form $(0,b,0,d)$ for some $b, d > 0$. Since its image $(0,b,0,d) \cdot \sigma = (b, 0, 0, b+d)$ belongs to  $(+,0,0,+)$, the corresponding arrow of the diagram is proven. 

\underline{3.} Consider the action by $\sigma^{-1}$ on $(-,0,0,+)$. Let 
$a<0$, $d>0$ then $(a,0,0,d) \cdot \sigma^{-1} = (a, d + (a-d)^{-}, 0, -(a-d)^{-}) = (a,a,0,-a+d) \in (-,-,0,+) \subset (-, -, +0, +)$.

\underline{4.} Consider the action by $\sigma$ on $(+,0,0,+)$. Let $a, d>0$ then $(a,0,0,d) \cdot \sigma =  (a, -a, 0, a+d) \in (+,-,0,+) \subset (+, -, -0, +)$.

\underline{5.} Consider the action by $\sigma^{-1}$ on $(0,+,-,0)$. Let $b>0$, $c<0$ then $(0,b,c,0) \cdot \sigma^{-1} = (-b - (-c)^{+}, 0, c - (c)^{-}, b - (-c)^{-}) = (-b+c, 0, 0, b) \in (-,0,0,+)$.

\underline{6.} Consider the action by $\sigma$ on $(0,+,+,0)$. Let $b,c>0$ then $(0,b,c,0) \cdot \sigma =  (b+c, 0, 0, b) \in (+,0,0,+)$.

\underline{7.} Consider the action by $\sigma$ on $(0,+,-,0)$. Let $b > 0$, $c < 0$ then $(0,b,c,0) \cdot \sigma = (b, c, c, b-c) \in (+, -, -, +) \subset (+, -, -0, +)$. 

\underline{8.} Consider the action by $\sigma^{-1}$ on $(0,+,+,0)$. Let $b, c > 0$ then $(0,b,c,0) \cdot \sigma^{-1} = (-b, -c, c, b+c) \in (-, -, +, +) \subset (-, -, +0, +)$. 

\underline{9.} Let us demonstrate that $(-, -, +0, +) \cdot \sigma^{-1} \in (-, -, +0, +)$. Since  $a < 0$, $b < 0$, $c \geqslant 0$, $d > 0$, the action is given by formulae
\begin{equation*}
\begin{cases} 
\begin{array}{l}
a''= a  - (a + b - c)^+,  \cr 
b'' = d +  (a + b - c - d)^-, \cr 
c'' = c - (-a + c + d)^-, \cr 
d'' = b - (a + b - c - d)^-. \cr 
\end{array}
\end{cases}
\end{equation*}
Moreover, $a + b - c < 0$ and $a+b-c-d < 0$ imply that $(a,b,c,d) \cdot \sigma^{-1} = (a, a + b - c, c,  -a + c + d) \in (-, -, +0, +)$. 

\underline{10.} Let us demonstrate that $(+0, +, -, -) \cdot \sigma^{-1} \in (-, -, +0, +)$.  Since $a\geqslant 0$, $b>0$, $c<0$, $d<0$, the action is given by formulae  
\begin{equation*}
\begin{cases} 
\begin{array}{l}
a'' = a - b - (a - c)^+,  \cr 
b'' = d + (a - c )^-, \cr 
c''  = c - d - (-a + c)^-, \cr 
d'' = b - (a - c)^-. \cr 
\end{array}
\end{cases}
\end{equation*}
Moreover, $a-c > 0$ implies that $(a,b,c,d) \cdot \sigma^{-1} = (-b+c, d, a-d, b) \in (-, -, +, +) \subset (-,-,+0,+)$. 

\underline{11.} Let us demonstrate that $(+0, +, -, -) \cdot \sigma \in (+, -, -0, +)$. Since $a \geqslant 0$, $b > 0$, $c < 0$, $d < 0$, the action is given by formulae
\begin{equation*}
\begin{cases} 
\begin{array}{l}
a' = a +b + (-a + c)^+,  \cr 
b' = d -  (a - c )^+, \cr 
c'  = c + d + (a - c)^-, \cr 
d' = b + (a - c)^+. \cr 
\end{array}
\end{cases}
\end{equation*}
Moreover, $a-c>0$ implies that $(a,b,c,d) \cdot \sigma = (a+b, d - a + c, c+d, b + a - c) \in (+, -, -, +) \subset (+, -, -0, +)$. 

\underline{12.} Let us demonstrate that  $ (+, -, -0, +) \cdot \sigma = (+, -, -0, +) $. Since $a>0$, $b<0$, $c \leqslant 0$, $d > 0$, the action is given by formulae
\begin{equation*}
\begin{cases} 
\begin{array}{l}
a' = a + (-a + b + c)^+,  \cr 
b' = d -  (a - b  - c  + d)^+, \cr 
c'  = c  + (a - c + d)^-, \cr 
d' = b + (a - b- c + d)^+. \cr 
\end{array}
\end{cases}
\end{equation*}
Moreover, $-a + b + c < 0$, $a - b - c + d > 0$, and $a - c + d > 0$ imply that $(a,b,c,d) \cdot \sigma = (a, -a + b + c, c, a - c + d) \in (+, -, -0, +)$. 

\underline{13.} Let us demonstrate that  $ (-0, +, +, -) \cdot \sigma \in (+, -, -0, +) $. Since $a \leqslant 0$, $b > 0$, $c > 0$, $d < 0$, the action is given by formulae
\begin{equation*}
\begin{cases} 
\begin{array}{l}
a'= a + b + (-a + c)^+,  \cr 
b'= d-  (a - c )^+, \cr 
c' = c + d + (a - c)^-, \cr 
d'= b + (a - c)^+. \cr 
\end{array}
\end{cases}
\end{equation*}
Moreover, $a-c<0$ implies that $(a,b,c,d) \cdot \sigma = (b+c, d, d+a, b) \in (+, -, -, +) \subset (+, -, -0, +)$. 

\underline{14.} Let us demonstrate that $(-0, +, +, -) \cdot \sigma^{-1} = (-, -, +0, +)$. Since $a \leqslant 0$, $b>0$, $c>0$, $d<0$, the action is given by 
\begin{equation}
\begin{cases} 
\begin{array}{l}
a'' = a - b - (a - c)^+,  \cr 
b'' = d +  (a - c )^-, \cr 
c''  = c- d - (-a + c)^-, \cr 
d''= b - (a - c)^-. \cr 
\end{array}
\end{cases}
\end{equation}
Moreover, $a-c<0$ implies that $(a,b,c,d) \cdot \sigma^{-1} = (a-b, d+a-c, c-d, b-a+c) \in (-,-,+,+) \subset (-,-,+0,+)$.  
The proof is completed. 
\end{proof}

\begin{remark}
Theorem~\ref{T:VB2} allows us to introduce on $VB_{2}$ various ``coordinate systems''. For example, taking $(0,2,0,1)$ as the initial vector, for any $\beta \in VB_{2}$ one can define $(0,2,0,1) \cdot \beta$ as its coordinates. In this sense Theorem~\ref{T:VB2} gives an anolog of Dynnikov coordinates originally defined for braid groups. 
\end{remark}

\begin{remark}
It is shown in the proof of Theorem~\ref{T:VB2} that the $VB_{2}$-action is faithful on any vector of the form $(0,x,0,y)$, where $x$ and $y$ are different positive integers. Note that the action on some other vectors of $\mathbb Z^4$ can fail to be faithful. For example, $(0,0,0,1) \cdot \sigma = (0,0,0,1)$ and $(0,1,0,1) \cdot \rho = (0,1,0,1)$. 
\end{remark}

\begin{remark} 
Let us define the norm of a quadruple  $|| (a, b, c, d) || = |a| + |b| + |c| + |d|$. Obviously, the norm is invariant under the $\rho$-action. One can easily see from the proof of Theorem~\ref{T:VB2} that all the arrows labelled $\sigma$ or $\sigma^{-1}$ in Figure \ref{F:VB2} increase the norm. For example, considering case 13, one gets $|| (a,b,c,d) \cdot \sigma || = || (b+c, d, d+a, b)  || = |b+c| + |d| + |d+a| + |b| > |a| + |b| + |c| + |d|$, because $a \leqslant 0$, $b>0$, $c>0$, $d<0$. 
However, this property doesn't hold if one takes an arbitrary vector from $\mathbb Z^{4}$: $|| (7,4,1,1) || = 13$, but $|| (7,4,1,1) \cdot \sigma || = || (3,1,6,1) || = 11$. %it is nice to have a concrete example
\end{remark}

\section*{Acknowledgements}

Work performed under the auspices of the Russian Foundation for Basic Research by the grant 10-01-00642 and by the French -- Russian  grant 10-01-91056. 
The first and the second named authors thank IRMAR,  Universit\'e de Rennes 1 and  Abdus Salam School of Mathematical Sciences, GC University Lahore for their hospitality.

\end{document}